\documentclass[reqno, 12pt]{article}

\pdfoutput=1

\usepackage{enumerate}
\usepackage{latexsym}
\usepackage[centertags]{amsmath}
\usepackage{amsfonts}
\usepackage{amssymb}
\usepackage{amsthm}
\usepackage{newlfont}
\usepackage{graphics}
\usepackage{color}
\usepackage{float}
\usepackage{diagbox}
\usepackage[linesnumbered,ruled,vlined]{algorithm2e}
\usepackage[pagebackref]{hyperref}
\usepackage{longtable}
\usepackage{rotating}
\usepackage{multirow}
\usepackage{extarrows}
\usepackage[sort,compress,numbers]{natbib}
\usepackage[utf8]{inputenc}

\newtheorem{thm}{Theorem}[section]
\newtheorem{cor}[thm]{Corollary}
\newtheorem{lem}[thm]{Lemma}
\newtheorem{prop}[thm]{Proposition}

\theoremstyle{mydefinition}
\newtheorem{dfn}[thm]{Definition}
\theoremstyle{myremark}
\newtheorem{rem}[thm]{Remark}
\newtheorem{exa}[thm]{Example}

\newtheorem{prob}[thm]{Problem}
\newtheorem{alg}[thm]{Algorithm}
\allowdisplaybreaks[4]
\SetKwInput{KwInput}{Input}                
\SetKwInput{KwOutput}{Output}              

\newcommand{\C}{{\mathbb{C}}}
\newcommand{\Rv}[1]{{#1}^\looparrowleft}

\newcommand{\R}{{\mathbb{R}}}

\title{An $O(n\log^2n)$ Algorithm for Computing Hankel Determinants up to Order $n$}

\author{Feihu Liu$^{\color{blue} \dag}$, Guoce Xin$^{\color{blue} \ddag}$, and Zihao Zhang$^{\color{blue} \S}$
\\[2mm]
{\small $^{\color{blue} \dag}$ Center for Combinatorics, LPMC}\\[-0.8ex]
{\small Nankai University, Tianjin 300071, P.R.~China}\\
{\small $^{\color{blue} \ddag}$ School of Mathematical Sciences,}\\[-0.8ex]
{\small Capital Normal University, Beijing, 100048, P.R.~China}\\
{\small $^{\color{blue} \S}$ School of Mathematics and Statistics}\\[-0.8ex]
{\small Beijing Institute of Technology, Beijing 102400, P.R.~China}\\
{\small {\color{blue} $^\dag$} Email address: liufeihu7476@163.com}\\
{\small {\color{blue} $^\ddag$} Email address: guoce\_xin@163.com}\\
{\small {\color{blue} $^\S$} Email address: zihao-zhang@foxmail.com}
}

\date{\today}

\begin{document}

\maketitle

\begin{abstract}  
Given the rational power series $h(x) = \sum_{i \geq 0} h_i x^i \in \C[[x]]$,  the Hankel determinant of order $n$ 
is defined as $H_n(h(x)) := \det (h_{i+j})_{1 \leq i,j \leq n}$.
We explore the relationship between the Hankel continued fraction and the generalized Sturm sequence. This connection inspires the development of a novel algorithm for computing the Hankel determinants $\{H_i(h(x))\}_{i=1}^{n}$ using $O(n \log^2 n)$ arithmetic operations.
We also explore the connection between the generalized Sturm sequences and the signature of Hankel matrices.
\end{abstract}

\noindent
\begin{small}
\end{small}

\noindent
\begin{small}
\end{small}

\section{Introduction}
Let $F=(h_0,h_1,h_2,\ldots)$ be a sequence with $h_i\in \C$, and denote by 
\begin{align*}
F(x)=h_0+h_1x+h_2x^2+\cdots \in \C[[x]]
\end{align*}
its generating function. 
Define the \emph{Hankel matrices} of order $n$ of $F(x)$ by 
\begin{align*}
\mathcal{H}_n(F(x))= (h_{i+j})_{0\leq i,j\leq n-1}
=\left ( \begin{matrix}
h_0 & h_1 & \cdots & h_{n-1} \\
h_1 & h_2 & \cdots & h_{n}  \\
\vdots & \vdots & \ddots & \vdots \\ 
h_{n-1} & h_{n} & \cdots & h_{2n-2} \\ 
\end{matrix} \right ).
\end{align*}
or abbreviated as $\mathcal{H}_n(F)$. 
The \emph{Hankel determinant} of order $n$ of $F(x)$ is defined by $H_n(F)=\det \mathcal{H}_n(F(x))$.

The computation of Hankel determinants is a well-studied topic in combinatorics and computer algebra. For instance, Hankel determinants for specific sequences have been explored in several works \cite{Krattenthaler05, chien2022hankel, cigler2011some, Sulanke-Xin}, among others. Quadratic forms associated with Hankel matrices offer a means for solving real root counting problems and real root localization problems; see \cite{parrilo2019algebraic}. Furthermore, certain properties of sequences of Hankel determinants provide the theoretical foundation upon which both Koenig's method and the Rutishauser \texttt{qd} algorithm, for the approximation of zeros and poles of meromorphic functions, are based; see \cite{gutknecht2011qd,householder1970numerical}.

In the context of algorithms, for a Hankel matrix of order \( n \), Sendra and Llovet \cite{Sendra93} have presented a method to compute Hankel determinants in \( O(n^2) \) arithmetic operations. For a more comprehensive overview of algorithms related to Hankel matrices, we refer the reader to the survey by Heinig and Rost \cite{HeinigRost}.
Kravanja and Van Barel \cite{BK2000} present a generically \( O(n \log^2 n) \) algorithm for computing the determinant of a nonsingular complex Hankel matrix whose size is a power of $2$. 
 In \cite{kravanja2000coupled, dietrich1996nlog2n}, there are two  \( O(n \log^2 n) \) algorithms for computing the determinants of nonsingular Hankel matrices, which leverage fast Fourier transform techniques and efficient interpolation methods.

These aforementioned algorithms appear to be designed for order \( n \). The continued fraction method is an essential tool for evaluating Hankel determinants and is particularly effective 
for calculating Hankel determinants starting from order $1$ up to order $n$. In this context, we utilize the following significant continued fraction, which was introduced by Han \cite{HanGuoNiu}.

\begin{dfn}\emph{\cite{HanGuoNiu}} 
For any positive integer $\delta$, a \emph{super continued fraction} associated with $\delta$, termed a \emph{super $\delta$-fraction}, is defined as
\begin{align}\label{eq_H-f}
F(x)=\cfrac{v_0x^{k_0}}{1+u_{1}(x)x
-\cfrac{v_1x^{k_0+k_1+\delta}}{1+u_{2}(x)x
-\cfrac{v_2x^{k_1+k_2+\delta}}{1+u_3(x)x-\ddots}}},
\end{align}
where $v_j \neq 0$ are constants, $k_j$ are nonnegative integers, and $u_j(x)$ are polynomials with degree at most $k_{j-1}+\delta-2$.
\end{dfn}

 Special cases of the super \( \delta \)-fraction, which have determinants formulas include the following: 
when all \( k_j = 0 \) and \( \delta = 1\) (or \( \delta = 2 \)),
it is the \emph{S-continued fraction} (or \emph{J-continued fraction}) \cite{Krattenthaler05,FlajoletDM}, and Hankel determinants of \( F(x) \) can be computed  if these determinants are nonzero for all \( n \);
when \( \delta = 1 \) and \( u_j(x) = 0 \), it becomes a \emph{C-continued fraction} and has determinants formula \cite{Cigler-C13, PaulBarry}; When \( \delta = 2 \), it is the \emph{H-continued fraction}\cite{HanGuoNiu}, which also relates to the S-X continued fraction methods \cite{Sulanke-Xin} and   
the Hankel determinants can be evaluated by following theorem whose existence and uniqueness are guaranteed without any condition.

\begin{thm}\emph{\cite[Theorem 2.2]{HanGuoNiu}}\label{HanGuoNiu-HankelCF}
Let $F(x)$ be a power series such that its H-continued fraction is given by \eqref{eq_H-f} with $\delta =2$. 
Then, all non-vanishing Hankel determinants of $F(x)$ are given by
$$H_{s_j}(F) = (-1)^{\epsilon} v_0^{s_j} v_1^{s_j -s_1} v_2^{s_j -s_2} \cdots v_{j-1}^{s_j -s_{j-1}},$$
where $\epsilon = \sum_{i=0}^{j-1} k_i (k_i +1) /2$  and $s_j = k_0 + k_1+ \cdots + k_{j-1} +j$ for every $j \geq 0$.   
\end{thm}

Utilizing the algorithm presented in \cite{sokal2023simple}, if $F(x)$ represents a power series, the $\delta$-continued fraction of $F(x) \bmod x^{2n-1}$ can be calculated with $O(n^2)$ field operations, subsequently yielding the Hankel determinants.
In this paper, we examine the relationship between the H-continued fraction and the generalized Sturm sequence, which yields a novel method for obtaining the H-continued fraction.
This inspires us to provide a new $O(n \log^2 n)$ algorithm \texttt{CompHD} for 
computing the  Hankel determinants $\{H_i(h(x))\}_{i=0}^{n-1}$.
It is worth mentioning that we do not require the non-singular condition as in the previous algorithms.

Our algorithm rely on the half-Greatest Common Divisor (half-GCD) algorithm 
which is an efficient method for computing the greatest common divisor
of two polynomials, see detail in \cite[Section 11]{ModernComputerA}. Comparing to the traditional GCD algorithm,
it employs a divide-and-conquer strategy to quickly calculate the all quotients without computing all the remainders.
By splitting the input polynomials into high and low-degree parts, the algorithm recursively computes the quotients of the high-degree parts and uses those quotients to adjust the low-degree parts. Te complexity of this algorithm is $O(n \log^2 n)$.


This paper is organized as follows.
In Section \ref{sec_Hfrac}, we present the connection between H-continued fraction and Sturm sequence.
In Section \ref{Section-33}, we present Algorithm \texttt{CompHD} for computing the Hankel determinants and give a formula for the signature of the nonsingular Hankel matrix.

\section{Hankel continued fraction and generalized Sturm sequence} \label{sec_Hfrac}

\begin{dfn} \label{def-generalizedSturmsequence}
Consider $f_0(x)$ and $f_1(x)$ as polynomials where $\deg(f_0) > \deg(f_1)$.
 The \emph{generalized Sturm sequence} $\{f_i(x)\}_{i=0}^{s+1}$ where $f_{s+1}(x) = 0$ is constructed as follows:
 We define the sequence by
\[
f_{i+2}(x) = - \emph{\texttt{rem}}(f_i(x), f_{i+1}(x)).
\]
where $\emph{\texttt{rem}}(f_i(x),f_{i+1}(x))$ denotes the remainder
obtained by dividing $f_i(x)$ by $f_{i+1}(x)$.

 We assume that the sequence $\{f_i(x)\}_{i=0}^{s+1}$ satisfies the relations:
\[ f_i(x) = B_i(x)f_{i+1}(x) - f_{i+2}(x), \quad 0 \leq i \leq s-1. \]
Furthermore, we denote $m_i = \deg(B_i)$ and $b_i$ represents the leading coefficient of $B_i(x)$.
Note that when  $f_1(x) = f'(x)$, it is called the (traditional) \emph{Sturm sequence} of $f(x)$. 
\end{dfn}

It is evident that 
\begin{align}\label{eq_degB}
\deg(f_i) = \deg(B_i) + \deg(f_{i+1}) \quad \text{and} \quad \deg(B_i) \geq 1.
\end{align}

Given a polynomial $f(x)$, we denote the transformation $\Rv f(x) = x^{\deg (f)} \cdot f(x^{-1})$, where ``$\Rv{} $" signifies the reversal of the polynomial.  For instance, if $f(x) = 3x^2 + 4x^3 + 6x^5$, then $\Rv f(x) = 6 + 4x^2 + 3x^3$.
Note that $\deg (\Rv f) = \deg(f) - \operatorname{ldeg} (f)$, where $\operatorname{ldeg}(f)$ denotes the lowest degree of the polynomial $f(x)$.

Thus, for $0\leq i\leq s-2$, we can verify that
\begin{align}\label{eq_g3}
\Rv f_i(x) = \Rv B_i(x) \cdot \Rv f_{i+1}(x) - x^{\deg(f_i) - \deg(f_{i+2})} \cdot \Rv f_{i+2}(x).
\end{align}
For $0\leq i\leq s-2$, by \eqref{eq_degB}, we have
\begin{align*}
\deg(f_i) - \deg(f_{i+2}) &= \deg(B_i) + \deg(f_{i+1}) - \deg(f_{i+2}) \\
&= \deg(B_i) + \deg(B_{i+1})\\
&=m_i+m_{i+1}.
\end{align*}
At the terminal step, since $f_{s+1}(x)=0$, we have
\[
f_{s-1}(x)=B_{s-1}(x)f_s(x),
\]
and hence
\[
\Rv f_{s-1}(x)=\Rv B_{s-1}(x)\Rv f_s(x).
\]
Rewriting these identities yields the following lemma.

\begin{lem} \label{lem_gcf}
Following the notations in Definition \ref{def-generalizedSturmsequence}. For $0\leq i\leq s-2$, we have
\begin{align}\label{Basic-Contin-Fract}
\frac{x^{m_i-1} \Rv f_{i+1}(x)}{\Rv f_i(x)}=
\cfrac{x^{m_i-1}}{\Rv B_i(x) - \cfrac{x^{(m_i-1)+(m_{i+1}-1)+2} \Rv f_{i+2}(x)}{\Rv f_{i+1}(x)}}.
\end{align}
For the terminal index $i=s-1$, we have
\[
\frac{x^{m_{s-1}-1}\Rv f_s(x)}{\Rv f_{s-1}(x)}
=
\frac{x^{m_{s-1}-1}}{\Rv B_{s-1}(x)}.
\]
Moreover, $\deg(\Rv B_i)\leq m_i$ for all $0\leq i\leq s-1$, and the constant term of $\Rv B_i(x)$ is nonzero for all $0\leq i\leq s-1$.
\end{lem}

Lemma \ref{lem_gcf} precisely establishes the H-continued fraction representation of
\[
S_{f_0,f_1}(x):=x^{m_0-1}\frac{\Rv f_1(x)}{\Rv f_0(x)},\qquad
m_0=\deg(f_0)-\deg(f_1).
\]
We denote this expression by $S_{f_0,f_1}(x)$ for brevity in the subsequent discussion. We now state the main result of this section.

\begin{thm}\label{thm-General-Case-H}
Following the notations in Definition \ref{def-generalizedSturmsequence}.
If $r=m_0+m_1+\cdots+m_{\kappa-1}$ for a certain $1\leq \kappa\leq s$, then we have
\begin{align*}
  H_r\left(S_{f_0,f_1}(x) \right)=(-1)^{{\sum_{i=0}^{\kappa-1} \frac{m_i(m_i-1)}{2}}}
\frac{1}{\prod_{i=0}^{\kappa-2}b_{i}^{2(r-\sum_{j=0}^{i-1}m_j)-m_i}\cdot b_{\kappa-1}^{m_{\kappa-1}}}.
\end{align*}
Otherwise, 
$ H_r\left(S_{f_0,f_1}(x)\right)=0.$
\end{thm}

\begin{proof}
Using  Lemma \ref{lem_gcf} recursively, we have the H-continued fraction of $S_{f_0,f_1}(x)$:
\begin{align}\label{eq_Hcf_Pf}
S_{f_0,f_1}(x)=\cfrac{ b_0^{-1} x^{m_0-1}}{(\Rv B_0(x)/b_0) 
-\cfrac{ (b_0 b_1)^{-1} x^{(m_0-1)+ (m_1-1) +2}}{(\Rv B_1(x)/b_1)
-\cfrac{ (b_1 b_2)^{-1}x^{(m_1 -1) + (m_2-1) +2}}{(\Rv B_2(x)/b_2) -\ddots}}}.
\end{align}
According to Theorem \ref{HanGuoNiu-HankelCF},  we have two cases.
If there is no $\kappa$ such that $r=m_0+m_1+\cdots +m_{\kappa-1}$, then 
$$H_r\left(S_{f_0,f_1}(x) \right)=0.$$
If $r=m_0+m_1+\cdots+m_{\kappa-1}$ for a certain $\kappa$, then we have
\begin{align*}
H_r\left(S_{f_0,f_1}(x)\right)&=(-1)^\epsilon
\frac{1}{(b_{0})^r \ (b_{0}b_{1})^{r-m_0} \cdots (b_{\kappa-2}b_{\kappa-1})^{r-\sum_{i=0}^{\kappa-2}m_i} }\\
&=(-1)^\epsilon
\frac{1}{\prod_{i=0}^{\kappa-2}b_{i}^{2(r-\sum_{j=0}^{i-1}m_j)-m_i}\cdot b_{\kappa-1}^{r-\sum_{i=0}^{\kappa-2}m_i}} ,
\end{align*}
where $\epsilon = {\sum_{i=0}^{\kappa-1}m_i(m_i-1)}/{2}$.

This completes the proof.
\end{proof}

This theorem also establishes the following recursive formula for the determinant
$H_r\left(S_{f_0,f_1}(x)\right)$.
\begin{cor} \label{cor-mainrec}
Following the notations in Definition \ref{def-generalizedSturmsequence}.
Let $0\leq \kappa\leq s-1$ and set
$r=m_0+\cdots+m_{\kappa-1}$, with the convention $r=0$ if $\kappa=0$.
Then $H_t(S_{f_0,f_1})=0$ for $r<t<r+m_\kappa$, and
\[
H_{r+m_\kappa}(S_{f_0,f_1})=
(-1)^{\frac{m_\kappa(m_\kappa-1)}{2}}
\left(b_\kappa\prod_{i=0}^{\kappa-1}b_i^2\right)^{-m_\kappa}
H_r(S_{f_0,f_1}),
\]
where the empty product is understood to be $1$.
The initial condition is $H_0(S_{f_0,f_1})=1$.
\end{cor}

\section{Two applications}\label{Section-33}

\subsection{Algorithm for Hankel determinants}
In this section, we address the following problem:
\begin{prob}
Compute the initial $n$ Hankel determinants for a specified rational power series $h(x)$.
\end{prob}

\begin{rem}
 To compute the first $n$ Hankel determinants for any given power series $r(x)$, we can define $h(x)$ be the polynomial $r(x) \bmod x^{2n-1}$.
\end{rem}

As established in Section \ref{sec_Hfrac}, if the polynomials $f_1(x)$, $f_0(x)$, and the rational power series $h(x)$ fulfill the criteria:
\begin{align}\label{eq_f1f0condition}
m_0= \deg(f_0) - \deg(f_1) > 0  \ \ \ \text{and} \ \ \  S_{f_0,f_1}(x)= x^{m_0-1} \frac{ \Rv f_1(x)}{\Rv f_0(x)}  = h(x) \bmod x^{2n-1},
\end{align}
then the generalized Sturm sequence of  $(f_0(x), f_1(x))$ yields the Hankel continued fraction  of $h(x)$ in the sense of modulo $x^{2n-1}$.

For a nonzero polynomial \(P(x)\), let \(\operatorname{ldeg}(P)\) denote
the lowest degree of \(P(x)\).

We now present a specific selection of $f_1(x)$ and $f_0(x)$ as outlined below.
\begin{lem}\label{lem-constructf0f1}
Let
\[
h(x)=\frac{N(x)}{D(x)}\in K[[x]],
\]
where \(N(x)\neq 0\) and \(D(0)\neq 0\). Define
\[
\nu=\deg N-\deg D+1.
\]
Construct \(f_0(x)\) and \(f_1(x)\) as follows:
\[
\left\{
\begin{array}{ll}
f_0(x)=\Rv D(x),&
f_1(x)=x^{-\nu}\Rv N(x),
\quad \text{if } \nu<0,\\[4pt]
f_0(x)=x^\nu\Rv D(x),&
f_1(x)=\Rv N(x),
\quad \text{if } \nu\geq 0.
\end{array}
\right.
\]
Then \(f_0(x)\) and \(f_1(x)\) meet the requirements of condition
\eqref{eq_f1f0condition}. Moreover, for \(\nu<0\), we have
\[
\deg(f_0)=\deg D,
\qquad
\deg(f_1)=\deg D-\operatorname{ldeg}(N)-1,
\]
and for \(\nu\geq 0\), we have
\[
\deg(f_0)=\deg N+1,
\qquad
\deg(f_1)=\deg N-\operatorname{ldeg}(N).
\]
If, in addition,
\[
\gcd(N(x),D(x))=1,
\]
then
\[
\gcd(f_0(x),f_1(x))=1.
\]
\end{lem}

\begin{proof}
Since \(D(0)\neq 0\), we have
\[
\deg(\Rv D)=\deg D,
\qquad
\deg(\Rv N)=\deg N-\operatorname{ldeg}(N).
\]
Verifying that $f_0(x)$ and $f_1(x)$ fulfill condition \eqref{eq_f1f0condition} is straightforward.

It remains to prove the last assertion. Suppose that
\[
\gcd(N(x),D(x))=1.
\]
Then \(N(x)\) and \(D(x)\) have no common roots. Consequently,  $x^{\deg(N)} N(x^{-1})$ and $x^{\deg(D)} D(x^{-1})$ also have distinct roots,  implying that $\gcd(f_0(x),f_1(x))=1$. This completes the proof.
\end{proof}

Kronecker's lemma states that a formal power series is rational if and only if its Hankel determinants are ultimately
zero \cite[p.5, Kronecker Lemma]{salem1963algebraic}. The following corollary gives a short proof of the rational-implies-eventual-vanishing direction.
\begin{cor}\label{cor-Hankelbound}
Let \( h(x)=\frac{N(x)}{D(x)}\) be a rational power series, where \(N(x)\neq0\), \(D(0)\neq0\), and \(\gcd(N(x),D(x))=1\). Let
\[
d= \max(\deg(N)+1 ,\deg(D)).
\]
Then $H_{d}(h(x)) \neq 0$ and $H_{t}(h(x)) = 0$ for all $t >d$.
\end{cor}

\begin{proof}
Let \( f_0(x) \) and \( f_1(x) \) be the polynomials as specified in Lemma \ref{lem-constructf0f1}, and consider \( \{f_i(x)\}_{i=0}^{s+1} \) as the generalized Sturm sequence. By Lemma \ref{lem-constructf0f1}, we have \(\deg(f_0)=d\). Since $\gcd(f_0(x),f_1(x))=1$, it is inferred that $f_s(x)$ is a nonzero constant. Applying Theorem \ref{thm-General-Case-H}, the fact that $f_s(x)$ is a nonzero constant implies
\[
\sum_{i=0}^{s-1} m_i=\deg(f_0)=d.
\]
Consequently, \({H}_{d}(h(x)) \neq 0 \) and \({H}_{r}(h(x)) = 0 \) for all \( r >d \).
\end{proof}

We summarize the preceding discussion in the following Hankel determinants algorithm. For a
polynomial \(P(x)\) and a positive integer \(L\), write
\[
P^{[L]}(x):=P(x)\bmod x^L .
\]

\begin{alg}[\texttt{CompHD}] \label{Algorithm-CHD}
$\newline$
Input: A rational power series
\[
h(x)=\frac{N(x)}{D(x)}\in K[[x]],
\qquad D(0)\neq 0,
\]
and a positive integer \(n\).

\noindent
Output: The Hankel determinant sequence
\[
H_1(h),H_2(h),\ldots,H_n(h).
\]

\begin{itemize}
  \item[\em 0.]
  Set \(L=2n-1\), and replace \(N(x)\) and \(D(x)\) by
  \[
  N^{[L]}(x)=N(x)\bmod x^L,
  \qquad
  D^{[L]}(x)=D(x)\bmod x^L.
  \]
  This truncation is valid because
\(
  h(x)\equiv \frac{N^{[L]}(x)}{D^{[L]}(x)}
  \pmod{x^L}.
\)
  In particular, if \(N^{[L]}(x)=0\), then return
\[
  H_1(h)=H_2(h)=\cdots=H_n(h)=0.
\]

  \item[\em 1.]
  Construct \(f_0\) and \(f_1\) from
  $N^{[L]}(x)$ and $D^{[L]}(x)$ as in Lemma \ref{lem-constructf0f1}.

  \item[\em 2.]
  Compute the quotients \(B_0,B_1,\ldots\) in the generalized Sturm
  sequence of \((f_0,f_1)\) until
  \[
  m_0+\cdots+m_j\geq \min(n,\deg f_0).
  \]
If the generalized Sturm sequence terminates before the cumulative degree
reaches \(n\), set all remaining Hankel determinants equal to zero.

\item[\emph{3}.]
Set \(H_0(h)=1\). Then use Corollary~\ref{cor-mainrec} to compute
\(H_i(h)\) recursively for \(1\le i\le n\).
\end{itemize}
\end{alg}

\begin{thm}
Let \( h(x) = {N(x)}/{D(x)} \) be a rational power series with
\(D(0)\neq0\). The Hankel determinant sequence
\[
\{H_i(h(x))\}_{i=1}^n
\]
can be computed in \(O(n\log^2 n)\) field operations using Algorithm
\ref{Algorithm-CHD}.
\end{thm}

\begin{proof}
Set \(L=2n-1 >0\). By Step 0 of Algorithm \ref{Algorithm-CHD}, we replace
\(N(x)\) and \(D(x)\) by
\[
N^{[L]}(x)=N(x)\bmod x^L,
\qquad
D^{[L]}(x)=D(x)\bmod x^L.
\]
Then
\[
\deg N^{[L]}<L,
\qquad
\deg D^{[L]}<L.
\]
Since \(D(0)\neq0\), we have
\[
h(x)\equiv \frac{N^{[L]}(x)}{D^{[L]}(x)}\pmod{x^L}.
\]
For \(1\le r\le n\), the determinant \(H_r(h)\) depends only on
the coefficients of \(h(x)\) up to degree \(2r-2\le 2n-2=L-1\).
Hence this truncation does not change \(H_1(h),\ldots,H_n(h)\).

By Lemma \ref{lem-constructf0f1}, the polynomials \(f_0(x)\) and
\(f_1(x)\) constructed from \(N^{[L]}(x)\) and \(D^{[L]}(x)\) satisfy
\[
h(x)=
x^{m_0-1}\frac{\Rv f_1(x)}{\Rv f_0(x)}
\pmod{x^{2n-1}}.
\]
Moreover, since \(\deg N^{[L]},\deg D^{[L]}<L\), the construction of
\(f_0(x)\) and \(f_1(x)\) requires \(O(2n-1)\)
operations.

The correctness of the remaining steps follows from Theorem
\ref{thm-General-Case-H} and Corollary \ref{cor-mainrec}. 

Regarding the complexity of Step 2, it is observed that Corollary \ref{cor-mainrec} and Theorem \ref{thm-General-Case-H} require only the degrees \(m_i\) and 
leading coefficients \(b_i\)
of the \(B_i(x)\)
rather than the full polynomials $f_i(x)$ for $i$ at most $n$. 
This is where the half-GCD algorithm comes into play, allowing the computation of $B_i$ in $O(n \log^2 n)$ field operations, as detailed in \cite[Section 11]{ModernComputerA}.

For the complexity of Step 3, within a single recursion, we initially require approximately two multiplications to construct
$\left(b_{\kappa+1} \prod_{i=0}^{\kappa} b_i^2 \right)$ from $ \left(b_{\kappa} \prod_{i=0}^{\kappa-1} b_i^2 \right)$. Additionally, we need approximately $\log (m_\kappa)$ steps to compute the $m_\kappa$-th power of $ \left(b_{\kappa} \prod_{i=0}^{\kappa-1} b_i^2 \right)$. Over the entire recursion, the total number of operations in the field is bounded by
$$\sum_{\kappa=0}^{s-1} 2+\log (m_\kappa) \leq 2s+ s \log (n/s) \leq 2s+ \frac{n}{e\ln(2)} \in O(n),$$
which confirms the claimed complexity.

\end{proof}

\begin{exa}
Given a rational power series $$h(x) = \frac{x^{2} \left(21 x^{14} + 15 x^{10}+x^{5}+6 x +4\right)}{x^{2}+3 x +1}.$$
We compute the Hankel determinants sequence $H_r(h(x))$ for $r=1\ldots 7$. The \emph{\texttt{CompHD}} proceeds as follows.
\end{exa}

\begin{proof}[Solution]

Step 0: We set
$
L=2\cdot 7-1=13.
$
The numerator and denominator obtained from the truncation of \(h(x)\) to order \(L\) are
\[
N^{[L]}(x)=x^{2}(15 x^{10}+x^{5}+6 x +4),
\qquad
D^{[L]}(x)=x^{2}+3x+1.
\]

Step 1: By Lemma \ref{lem-constructf0f1}, we construct
\[
f_0(x)=x^{11}(x^2+3x+1), \qquad
f_1(x)=4x^{10}+6x^9+x^5+15.
\]

Step 2: To compute the Hankel determinants up to $H_7(h(x))$, we compute
$B_0(x),B_1(x),\ldots$ successively until
\[
\sum_{i=0}^{t}m_i\geq 7,
\]
where $m_i=\deg(B_i)$. In this example, this requires
$B_0(x),\ldots,B_4(x)$. A direct computation gives
\begin{align*}
  B_0(x) &= \frac{1}{4}x^{3}+\frac{3}{8}x^{2}-\frac{5}{16}x+\frac{15}{32},
  \qquad
  B_1(x) = \frac{64}{45}x+\frac{4064}{2025},\\
  B_2(x) &= \frac{91125}{33536}x-\frac{2492775}{4393216},
  \qquad
  B_3(x) = -\frac{143877824}{61509375}x-\frac{254394664}{184528125},\\
  B_4(x) &= \frac{2767921875}{207033444463}x-\frac{2984742421875}{61281899561048}.
\end{align*}
Thus
\[
m_0=3,\qquad m_1=m_2=m_3=m_4=1,
\]
and the cumulative degrees are
\[
3,\ 4,\ 5,\ 6,\ 7.
\]

Step 3: Let $b_i=\operatorname{lc}(B_i)$. Using Corollary \ref{cor-mainrec} with the initial convention
\[
H_0(h(x))=1.
\]

Since $m_0=3$, we have
we first obtain
\[
H_1(h(x))=H_2(h(x))=0,
\]
and
\[
H_3(h(x))
=
(-1)^{\frac{3(3-1)}{2}}b_0^{-3}H_0(h(x))
=
-\left(\frac{1}{4}\right)^{-3}
=
-64.
\]
Since $m_1=1$, we get
\[
H_4(h(x))
=
\left(b_1b_0^2\right)^{-1}H_3(h(x))
=
\left(\frac{64}{45}\left(\frac{1}{4}\right)^2\right)^{-1}(-64)
=
-720.
\]
Similarly, using $B_2(x)$, $B_3(x)$, and $B_4(x)$ successively, we obtain
\begin{align*}
H_5(h(x))
&=
\left(b_2b_0^2b_1^2\right)^{-1}H_4(h(x))
=
-2096,\\
H_6(h(x))
&=
\left(b_3b_0^2b_1^2b_2^2\right)^{-1}H_5(h(x))
=
960,\\
H_7(h(x))
&=
\left(b_4b_0^2b_1^2b_2^2b_3^2\right)^{-1}H_6(h(x))
=
14060.
\end{align*}
\end{proof}

\subsection{The signature of Hankel matrix and the sign variation of generalized Sturm sequence}

Consider a nonsingular symmetric matrix \( A \in \mathbb{R}^{n \times n} \) with  all nonzero sequential principal minors 
\[ \Delta_0 = 1, \Delta_1, \ldots, \Delta_n. \]
 Define \( \epsilon_v \) as the number of  sign variations
  and \( \epsilon_p \) as the number of  sign permanences in $\{\Delta_i\}_{i=0}^{n}$. 
It is known that the signature \(\epsilon_A\) of \(A\) is
\(\epsilon_p-\epsilon_v\)
. But this formula becomes inapplicable when some sequential principal minors are zero.
However, Frobenius demonstrated that a signature formula exists for Hankel matrices in the general case.

\begin{lem}{\em\cite[Section 10]{gantmakher2000theory}} 
\label{lem-sigHankel}
For nonsingular Hankel matrix $H \in \R^{n \times n}$ with sequential principal minors
$$H_0=1,H_1,\ldots,H_n.$$
Suppose $H_{k} \neq 0$ only when $ k =r_0, \ldots ,r_s$,  and define
$$\epsilon_{r_{i-1}, r_{i}} = \begin{cases} 
(-1)^{\frac{r_{i}-r_{i-1}-1}{2} }  \ \emph{sign} \left( \frac{ H_{r_{i}}}{H_{r_{i-1}}} \right) & \text{if } r_{i}-r_{i-1} \text{ is odd;}\\
 0 & \text{otherwise.}  
\end{cases}$$
Then the signature  $\epsilon_H=\sum_{i=1}^{s} \epsilon_{r_{i-1}, r_{i}}$.
\end{lem}

For \(0\le t\le s-1\) and \(c\in\mathbb R\), let
\(\mathcal V_t(c)\) denote the number of sign variations, ignoring zeros,
in the finite sequence
\[
(f_0(c),f_1(c),\ldots,f_{t+1}(c)).
\]
We define
\[
\mathcal V_t(+\infty)=\lim_{c\to+\infty}\mathcal V_t(c),
\qquad
\mathcal V_t(-\infty)=\lim_{c\to-\infty}\mathcal V_t(c).
\]

By Lemma \ref{lem-sigHankel} and Corollary \ref{cor-mainrec}, we obtain the following theorem:

\begin{thm}\label{thm-signature} 
Let \(h(x)\in \mathbb R[[x]]\), and let \(f_0(x),f_1(x)\in \mathbb R[x]\)
satisfy condition \eqref{eq_f1f0condition}. Let
\[
f_i(x)=B_i(x)f_{i+1}(x)-f_{i+2}(x),\quad
m_i=\deg B_i,\quad b_i=\operatorname{lc}(B_i),\ (\text{the leading coefficient of } B_i(x).)
\]
be the associated generalized Sturm sequence data as in Definition
\ref{def-generalizedSturmsequence}. Suppose that the Hankel matrix
\(\mathcal H_n(h(x))\) is nonsingular. Then there exists
\(0\le t\le s-1\) such that
\[
\sum_{i=0}^{t}m_i=n.
\]
Moreover,
\[
\epsilon_{\mathcal H_n(h)}
=
\mathcal V_t(-\infty)-\mathcal V_t(+\infty)
=
\sum_{\substack{0\le i\le t\\ m_i\ \mathrm{odd}}}
\operatorname{sign}(b_i).
\]
\end{thm}

\begin{proof}
By Theorem \ref{thm-General-Case-H}, since $H_n(h(x)) \neq 0$, there exists $0 \leq t \leq s-1$ such that $\sum_{i=0}^{t} m_i = n$.
Set
\[
r_i=m_0+\cdots+m_{i-1},\qquad r_0=0.
\]
Then Corollary~\ref{cor-mainrec} gives
\[
\operatorname{sign}\left(
\frac{H_{r_i+m_i}(h(x))}{H_{r_i}(h(x))}
\right)
=
(-1)^{m_i(m_i-1)/2}\operatorname{sign}(b_i^{m_i}).
\]

Applying Lemma \ref{lem-sigHankel}, we find
\[
 \epsilon_{\mathcal{H}_n(h(x))}  = \sum_{i=0,\ m_i \text{ is odd} }^{t} (-1)^{\frac{ m_i-1}{2}} \text{sign}\left( \frac{H_{r+ m_{i}}(h(x))}{H_{r}(h(x))} \right) = \sum_{i=0, m_i \text{ odd}}^{t} \text{sign}(b_{i}).
\]
On the other hand, let \(a_i x^{\deg(f_i)}\) be the leading term of
\(f_i(x)\). Since
\[
b_i x^{m_i}
=
\frac{a_i x^{\deg(f_i)}}{a_{i+1}x^{\deg(f_{i+1})}},
\]
the sign change between \(f_i(x)\) and \(f_{i+1}(x)\) as
\(x\to\pm\infty\) is determined by the sign of \(b_i x^{m_i}\).
Therefore,
\[
\mathcal V_t(+\infty)
=
\#\{0\le i\le t:\ b_i<0\},
\]
and
\[
\mathcal V_t(-\infty)
=
\#\{0\le i\le t:\ b_i<0,\ m_i\ \text{even}\}
+
\#\{0\le i\le t:\ b_i>0,\ m_i\ \text{odd}\}.
\]
Hence
\[
\mathcal V_t(-\infty)-\mathcal V_t(+\infty)
=
\sum_{\substack{0\le i\le t\\ m_i\ \mathrm{odd}}}
\operatorname{sign}(b_i).
\]
This completes the proof.
\end{proof}

For the  real roots counting problem, the Sturm's Theorem (see, for instance, \cite[Page 419]{Knuth81} or \cite{Sturm29}) states that the number of real roots of a polynomial can be characterized by the pattern of sign changes in its Sturm sequence.
\begin{thm}{\em \cite{Sturm29}}\label{Real-rooted-one}
Let $f(x)$ be a univariate polynomial with the real coefficients.  Let $a,b\in \mathbb{R}$ with $a<b$, $f(a),f(b)\neq 0$. 
If $f(x)$ has distinct roots (i.e., $\gcd(f(x), f'(x))=1$),  then the number of real roots of $f(x)$ in $(a,b]$ is $\mathcal{V}(a)-\mathcal{V}(b)$ on the Sturm sequence of $f(x)$.
\end{thm}

On the other hand, as early as 1857, the quadratic forms associated with Hankel matrices 
provided a means for solving real root counting problems and real root localization problems; see \cite{parrilo2019algebraic} and \cite{krein1981method}.

Consider a polynomial $f(x) =a_n x^n + a_{n-1}x^{n-1} + \cdots + a_1x + a_0$, 
where the coefficients $a_i \in \R$  and $a_n =1$. By the fundamental theorem of algebra, $f(x)$ has exactly $n$ roots, denoted by $\{s_i\}_{i=1}^n$, in the complex field $\mathbb{C}$.
Let $p_k = \sum_{i=1}^n s_i^k$ represent the power sums of the roots and $P_{f}(x) = \sum_{k=0}^{2n-2} p_k x^k$.

The relationship between $f(x)$ and $P_f(x)$ has been established as follows:

\begin{prop}\label{prop-Pf}
Let
\[
f(x)=\prod_{i=1}^{n}(x-s_i),
\qquad
p_k=\sum_{i=1}^{n}s_i^k,
\]
and define
\[
P_f(x)=\sum_{k=0}^{2n-2}p_kx^k.
\]
Let \(f_0(x)=f(x)\) and \(f_1(x)=f'(x)\). Then
\[
\frac{\Rv f_1(x)}{\Rv f_0(x)}
\equiv
P_f(x)
\pmod{x^{2n-1}}.
\]
\end{prop}

\begin{proof}
Since
\[
\Rv f_0(x)=\prod_{i=1}^{n}(1-s_ix),
\]
we get 
\begin{align*}
\ln \frac{1}{\Rv f_0(x)}= \sum_{i=1}^n\ln \frac{1}{1-s_ix}=\sum_{i=1}^n\sum_{k\geq 1}\frac{1}{k}(s_ix)^k
=\sum_{k\geq 1}\frac{1}{k}p_kx^k.
\end{align*}
By taking the $x \frac{\mathrm{d}}{\mathrm{d} x} $ on both sides of the equation, we obtain 
\[
-\frac{x(\Rv f_0(x))'}{\Rv f_0(x)}
=
\sum_{i=1}^{n}\frac{s_ix}{1-s_ix}
=
\sum_{k\geq1}p_kx^k.
\]
On the other hand, since \(f_1(x)=f_0'(x)\), we have
\[
x(\Rv f_0(x))'
=
n\Rv f_0(x)-\Rv f_1(x).
\]
Hence
\[
\frac{\Rv f_1(x)}{\Rv f_0(x)}
=
n-\frac{x(\Rv f_0(x))'}{\Rv f_0(x)}
=
n+\sum_{k\geq1}p_kx^k.
\]
Since \(p_0=n\), it follows that
\[
\frac{\Rv f_1(x)}{\Rv f_0(x)}
=
\sum_{k\geq0}p_kx^k.
\]
Taking the truncation modulo \(x^{2n-1}\), we obtain
\[
\frac{\Rv f_1(x)}{\Rv f_0(x)}
\equiv
\sum_{k=0}^{2n-2}p_kx^k
=
P_f(x)
\pmod{x^{2n-1}}.
\]
Then the proposition follows.
\end{proof}

The first application of the law of inertia of quadratic forms to the investigation of the roots of algebraic equations is found in 
\cite{borchardt1847developpements}, where 
Borchardt deduced the following theorem postulated that all sequential principal minors of $\mathcal{H}_n(P_f)$ are nonsingular. 
Here, we give a short proof of this theorem without the necessity of Borchardt's assumption.
\begin{thm}{\em \cite{parrilo2019algebraic}}\label{Sign=Numb} 
If $f(x)$ has distinct roots (i.e., $\gcd(f(x), f'(x))=1$),
The signature of $\mathcal{H}_n(P_f(x))$ is equivalent to the number of real roots of $f(x)$. 
\end{thm}
\begin{proof}
By Corollary \ref{cor-Hankelbound}, note that $\gcd(f_0(x),f_1(x))=1$ and $\max (\deg(f) , \deg(f')+1) = n$, we have $\det \mathcal{H}_n(P_f) \neq 0$.
Then the theorem follows by Theorem \ref{thm-signature} and Sturm's theorem \ref{Real-rooted-one}. 
\end{proof}




\noindent
{\small \textbf{Acknowledgements:} 

The authors are grateful to Raphael Clifford for his careful reading and insightful suggestions, which have greatly improved the readability of this manuscript. This work was partially supported by the National Natural Science Foundation of China [12571355, 12071311].


\end{document}